\documentclass[11pt]{article}

\usepackage[a4paper,margin=31mm]{geometry}

\usepackage{amsmath,amssymb,amsthm,mathtools}
\usepackage{booktabs,array}
\usepackage{tikz}
\usetikzlibrary{arrows.meta,positioning,decorations.markings}
\usepackage[colorlinks=true,linkcolor=blue!55!black,citecolor=blue!55!black,urlcolor=blue!55!black]{hyperref}
\usepackage{enumitem}
\usepackage{microtype}
\hypersetup{pdftitle={Alternating Hamilton Paths and Cycles in Poset Products}}
\newtheorem{theorem}{Theorem}[section]
\newtheorem{lemma}[theorem]{Lemma}
\newtheorem{proposition}[theorem]{Proposition}
\newtheorem{corollary}[theorem]{Corollary}
\theoremstyle{definition}

\newtheorem{conj}[theorem]{Conjecture}
\newtheorem{remark}[theorem]{Remark}

\newcommand{\La}{\operatorname{La}}

\title{Alternating Hamilton Paths and Cycles in Poset Products}
\author{Yuxian Dong, Jianxi Mao
\thanks{
	Corresponding author.
    \newline\hspace*{3mm}
    {\it Email address:}\quad
    yuxiand@hotmail.com (Yuxian Dong); maojx@dlut.edu.cn (Jianxi Mao)
    }}
\date{\footnotesize
 School of Mathematical Sciences, Dalian University of Technology, Dalian 116024, P.R. China\\
}
\begin{document}

\maketitle

\begin{abstract}
We study alternating Hamilton paths and cycles in the comparability graphs of finite posets.
We prove that if a finite poset $P$ with at least two elements
admits an alternating Hamilton path, then $P\times B_r$
admits an alternating Hamilton cycle for every $r\ge2$,
where $B_r$ is the Boolean lattice.
As a corollary,
we settle a conjecture of Katona [Discrete Appl. Math., 2013].
Furthermore, we characterize the divisor lattices admitting alternating Hamilton paths and cycles.
\end{abstract}



\section{Introduction}
The \emph{Boolean lattice} $B_n$ is the poset of all subsets of $[n]=\{1,\ldots,n\}$, ordered by inclusion.
A \emph{Hamilton cycle} is a cycle that visits every vertex exactly once.
Hamilton cycles in the cover graph of $B_n$ correspond to cyclic binary Gray codes; see Savage~\cite{Savage} for a survey.
M\"utze~\cite{Mutze} proved that for every $k\ge1$, 
the subgraph of the cover graph of $B_{2k+1}$ induced by the two middle levels admits a Hamilton cycle,
confirming a conjecture of Havel~\cite{Havel} and Buck and Wiedemann~\cite{BuckWiedemann}.
Katona~\cite{Katona} investigated alternating Hamilton cycles in $B_n$
when studying snake-free families.
In this paper, we study alternating Hamilton paths and cycles
in poset products, in particular, products of a poset with a Boolean lattice or a chain.

Let $P$ be a finite poset and let $|P|$ denote its cardinality.
The \emph{comparability graph} of $P$ has vertex set $P$, with two distinct vertices adjacent if and only if they are comparable in $P$.
For example, the comparability graph of a chain  with $|P|$ elements is the complete graph $K_{|P|}$.
All paths and cycles considered below are taken in comparability graphs.
An~\emph{alternating Hamilton path} (for short, \emph{AHP}) in $P$ is an ordering $x_1,\ldots,x_{|P|}$ of all elements of $P$ such that
\[
x_1<x_2>x_3<x_4>\cdots \qquad \textrm{or} \qquad x_1>x_2<x_3>x_4<\cdots.
\]
By convention, a singleton admits an alternating Hamilton path.
An \emph{alternating Hamilton cycle} (for short, \emph{AHC}) in $P$ with $|P|\ge 4$ is a cyclic ordering of all the elements such that
\[
 x_1<x_2>x_3<\cdots<x_{|P|}>x_1\qquad \textrm{or} \qquad x_1>x_2<x_3>\cdots >x_{|P|}<x_1.
\]
Clearly, if $P$ admits an AHC, then it admits an AHP.
Katona conjectured the following.

\begin{conj}[{\cite[Conjecture in Section~8]{Katona}}]
\label{conj:katona}
For every $n\ge3$, $B_n$ admits an AHC.
\end{conj}

For finite posets $P$ and $Q$, their \emph{product} $P\times Q$ is the set of pairs $(p,q)$ with 
$p\in P$ and $q\in Q$,
ordered by $(p,q)\le(p',q')$ if and only if $p\le p'$ in $P$ and $q\le q'$ in $Q$.
Note that $B_{n}\times B_{m}\cong B_{n+m}$.
We now present our first result on AHPs and AHCs in products with Boolean lattices.

\begin{theorem}
\label{thm:intro-boolean}
Let \(P\) be a finite poset with $|P|\ge 2$.
\begin{enumerate}[label=(\roman*), font=\upshape]
\item If $P$ admits an AHP, then $P\times B_r$ admits an AHC for every $r\ge2$.
\item If $P$ admits an AHC, then $P\times B_r$ admits an AHP for every $r\ge0$.
\end{enumerate}
\end{theorem}

Since $B_1$ admits an AHP, taking $P=B_1$ and $r=n-1\ge2$ in Theorem~\ref{thm:intro-boolean} (i),
we obtain that $B_n\cong B_1\times B_{n-1}$ admits an AHC for every $n\ge3$.
This proves Conjecture~\ref{conj:katona}.
\begin{corollary}\label{coro:katona}
Conjecture~\ref{conj:katona} is true.
\end{corollary}

It is known~\cite[Section~3.3]{EC1} that Boolean lattices arise as special cases of divisor lattices.
For a positive integer $n$, the \emph{divisor lattice} $D_n$ is the poset of positive divisors of $n$, ordered by divisibility.
Every divisor lattice is isomorphic to a product of chains.

For $m\ge0$, let $C_m$ denote the chain poset with $m+1$ elements $1<2<\cdots<m+1$, where $C_0$ is a singleton. 
Suppose that
$n=p_1^{m_1}\cdots p_d^{m_d}$, where $p_1,\ldots,p_d$ are distinct primes and $m_1,\ldots,m_d$ are nonnegative integers. Then
$$
D_n\cong C_{m_1}\times\cdots\times C_{m_d}.
$$
In particular, when $m_1=\cdots=m_d=1$, we have $D_n\cong B_d$.
We refer the reader to~\cite{Anderson, EC1} for more details. 
We now establish the following result on AHPs and AHCs in products with chains.

\begin{theorem}
\label{thm:intro-chains}
Let $P$ be a finite poset.
\begin{enumerate}[label=(\roman*), font=\upshape]
\item If $P$ admits an AHP, then $P\times C_m$ admits an AHP for every $m\ne1$.
\item If $P$ admits an AHP, then $P\times C_m$ admits an AHC for every odd $m\ge3$.
\item If $P$ admits an AHC, then $P\times C_m$ admits an AHC for every $m\ne1$.
\item If $P$ admits an AHC, then $P\times C_m$ admits an AHP for every $m\ge0$.
\end{enumerate}
\end{theorem}

We now present the following result on AHPs and AHCs in divisor lattices.
\begin{theorem}
\label{thm:intro-products}
For a positive integer $n$, let $D_n$ be the divisor lattice. Then 
\begin{enumerate}[label=(\roman*), font=\upshape]
\item $D_n$ admits an AHP.
\item $D_n$ admits an AHC if and only if $|D_n|$ is even, $D_n\ncong C_1$ and $D_n\ncong B_2$.
\end{enumerate}
\end{theorem}

AHPs and AHCs are closely related to snake-free families. 
For finite posets $R$ and $P$, a \emph{weak embedding} of $R$ into $P$ is an injective map $f:R\to P$ such that $r<r'$ implies $f(r)<f(r')$.
We write $R\hookrightarrow P$ when such a map exists.
We say that a poset $P$ is $R$-free if there is no weak embedding of $R$ into $P$.
For results on subposet-free problems in Boolean lattices, see
\cite{Bukh,GriggsLu,KatonaXiao,ZhanLuZhao}.

For $k\ge2$, the \emph{$k$-element snake} $S_k$ is the poset on $\{x_1,\ldots,x_k\}$, whose only strict order relations are
\begin{equation}\label{def:snake}
x_1<x_2>x_3<x_4>\cdots.
\end{equation}
We take $S_1$ to be a singleton.
For results concerning snake-free families, see
\cite{Bukh,GriggsKatona,GriggsLu,Katona,KatonaTarjan,Sperner}.
Clearly, if $P$ admits an AHP and $P$ is self-dual (see Section~\ref{sec:boolean} for the definition),
then $S_k\hookrightarrow P$ for every $1\le k\le |P|$.
Katona~\cite[Section~8]{Katona} suggested that $S_k\hookrightarrow B_n$ whenever $k\le2^n$.
Since $B_n$ admits an AHP by Theorem~\ref{thm:intro-products} and $B_n$ is self-dual,
we have $S_k\hookrightarrow B_n$ for every $1\le k\le2^n$,
confirming Katona's statement.

The paper is organized as follows.
In Section~\ref{sec:boolean}, we give a constructive proof of Theorem~\ref{thm:intro-boolean}.
In Section~\ref{sec:chains}, we introduce a chain construction that preserves AHPs and AHCs and use it to prove Theorem~\ref{thm:intro-chains}.
In Section~\ref{sec:products}, we combine these results to prove Theorem~\ref{thm:intro-products}, and determine when a weak embedding of a snake into a divisor lattice exists.
Finally, in Section~\ref{sec:applications}, we construct snake-free families in divisor lattices and derive lower bounds for the maximum size of such families.

\section{Products with Boolean lattices}
\label{sec:boolean}
In this section, we prove Theorem~\ref{thm:intro-boolean}.
We first recall some basic properties of AHPs and AHCs.
Throughout the paper, all posets considered are finite and nonempty.
Consecutive elements in an AHP or AHC are comparable, but need not be related by a cover relation.

We call $R$ a \emph{refinement} of $P$ if they have the same underlying set and $x\le_P y$ implies $x\le_R y$ for all $x,y\in R$.
The \emph{dual} $P^*$ has the same underlying set as $P$, with $x\le_{P^*}y$ if and only if $y\le_P x$.
We call $P$ \emph{self-dual} if $P\cong P^*$.
The following results are immediate from the definitions.
\begin{lemma}
\label{lem:basic-properties}
Let $P$ be a finite poset.
\begin{enumerate}[label=(\roman*), font=\upshape]
\item If $P$ admits an AHC, then it admits an AHP.
\item If $P$ admits an AHC, then $|P|$ is even and $|P|\ge4$.
\item $P$ admits an AHP (resp., AHC) if and only if $P^*$ admits an AHP (resp., AHC).
\item If $P$ admits an AHP (resp., AHC), then so does every refinement of $P$.
\end{enumerate}
\end{lemma}

We now prove that if $|P|\ge2$ and $P$ admits an AHP, then $P\times B_2$ and $P\times B_3$ admit an AHC.

\begin{lemma}
\label{lem:b2}
Let $P$ be a finite poset with $|P|\ge 2$.
If $P$ admits an AHP, then $P\times B_2$ admits an AHC.
\end{lemma}

\begin{proof}
If $|P|=2$, then $P\cong B_1$ is a two-element chain and
$P\times B_2\cong B_3$, which admits the AHC
\[
\varnothing<\{1,2,3\}>\{1\}<\{1,2\}>\{2\}
<\{2,3\}>\{3\}<\{1,3\}>\varnothing.
\]
We now prove the case $|P|\ge3$. 
First, we suppose that the AHP in $P$ is of the form
\begin{equation}\label{x1-xp}
x_1<x_2>x_3<x_4>\cdots.
\end{equation}

\smallskip
\noindent\textbf{Step 1. Construct four disjoint alternating paths.}
Let $(x_2,x_3,\ldots,x_{|P|-1})$ be an alternating path of $P$ defined in~\eqref{x1-xp}. 
For each $A\in \{\varnothing, \{1\},\{2\}, \{1,2\}\}$, let
\[
I_A=(x_2,A),(x_3,A),\ldots,(x_{|P|-1},A).
\]
Clearly, $I_A$ is an alternating path in $P\times B_2$ from $(x_2,A)$ to $(x_{|P|-1},A)$.
We denote by $\overleftarrow{I_A}$ the path obtained by traversing $I_A$ in reverse order, 
$$
\overleftarrow{I_A}=(x_{|P|-1},A), (x_{|P|-2},A),\ldots,(x_{2},A).
$$
These four paths $I_{\varnothing}, I_{\{1\}},I_{\{2\}}, I_{\{1,2\}}$ are disjoint and traverse
$\{x_2,\ldots,x_{|P|-1}\}\times B_2$.
When $|P|=3$, each $I_A$ consists of a single vertex.

\smallskip
\noindent\textbf{Step 2. The case that $|P|$ is odd.}
Then~\eqref{x1-xp} becomes
$$
x_1<x_2>x_3<x_4>\cdots <x_{|P|-1}>x_{|P|}.
$$
Thus, we construct two joining paths whose endpoints lie in
$\{x_2\}\times B_2$ and whose internal vertices lie in
$\{x_1\}\times B_2$:
\begin{align}
(x_2,\varnothing)
>&(x_1,\varnothing)<(x_2,\{1\}),\label{1-x_1-x_2}\\
(x_2,\{2\})
>(x_1,\{2\})<&(x_1,\{1,2\})
>(x_1,\{1\})<(x_2,\{1,2\}).\label{2-x_1-x_2}
\end{align}
Similarly, we construct two joining paths whose endpoints lie in
$\{x_{|P|-1}\}\times B_2$ and whose internal vertices lie in
$\{x_{|P|}\}\times B_2$:
\begin{align*}
(x_{|P|-1},\varnothing)
>&(x_{|P|},\varnothing)<(x_{|P|-1},\{2\}),\\
(x_{|P|-1},\{1\})
>(x_{|P|},\{1\})<&(x_{|P|},\{1,2\})
>(x_{|P|},\{2\})<(x_{|P|-1},\{1,2\}).
\end{align*}
The four paths visit each vertex $(x_1,A)$ and $(x_{|P|},A)$ 
exactly once. 
See Figure~\ref{fig:b2-odd-joining}.

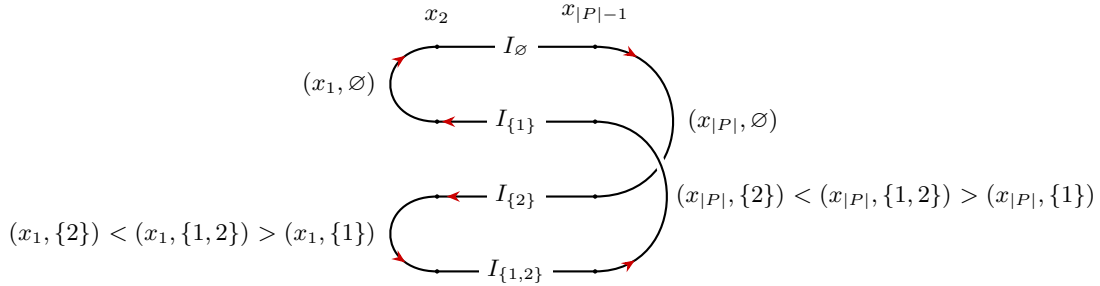
\begin{figure}[htbp]
\centering
\begin{tikzpicture}[
  scale=0.55,
  x=1cm,y=1cm,
  every node/.style={font=\footnotesize},
  every path/.style={line width=0.8pt},
  path label/.style={fill=white,inner sep=3pt},
  joining label/.style={fill=white,inner sep=2pt,align=left},
  traversal/.style args={#1}{
    postaction={decorate},
    decoration={markings,
      mark=at position #1 with {
        \arrow[red!80!black]{Stealth[length=5pt,width=4pt]}
      }
    }
  }
]
  \coordinate (L0) at (0,0);
  \coordinate (L1) at (0,-1.8);
  \coordinate (L2) at (0,-3.6);
  \coordinate (L3) at (0,-5.4);
  \coordinate (R0) at (3.8,0);
  \coordinate (R1) at (3.8,-1.8);
  \coordinate (R2) at (3.8,-3.6);
  \coordinate (R3) at (3.8,-5.4);

  \draw (L0)--(R0)
    node[midway,path label]
    {$I_\varnothing$};
  \draw[traversal=0.97] (R1)--(L1)
    node[midway,path label]
    {$I_{\{1\}}$};
  \draw[traversal=0.94] (R2)--(L2)
    node[midway,path label]
    {$I_{\{2\}}$};
  \draw (L3)--(R3)
    node[midway,path label]
    {$I_{\{1,2\}}$};

  \draw[traversal=0.75]
    (L1) .. controls (-1.5,-1.8) and (-1.5,0) .. (L0)
    node[pos=0.5,left=3pt,joining label]
    {$(x_1,\varnothing)$};
  \draw[traversal=0.75]
    (L2) .. controls (-1.5,-3.6) and (-1.5,-5.4) .. (L3)
    node[pos=0.5,left=3pt,joining label]
    {$(x_1,\{2\})$
     $<(x_1,\{1,2\})$
     $>(x_1,\{1\})$};

  \draw[traversal=0.18]
    (R0) .. controls (6.3,0) and (6.3,-3.6) .. (R2)
    node[pos=0.5,right=3pt,joining label]
    {$(x_{|P|},\varnothing)$};
  \draw[preaction={draw=white,line width=4pt},traversal=0.18]
    (R3) .. controls (6.1,-5.4) and (6.1,-1.8) .. (R1)
    node[pos=0.5,right=1pt,joining label]
    {$(x_{|P|},\{2\})$
     $<(x_{|P|},\{1,2\})$
     $>(x_{|P|},\{1\})$};

  \foreach \i in {0,1,2,3}{
    \fill (L\i) circle (1.5pt);
    \fill (R\i) circle (1.5pt);
  }
  \node[above=5pt] at (L0) {$x_2$};
  \node[above=5pt] at (R0) {$x_{|P|-1}$};
\end{tikzpicture}
\caption{The connections for odd $|P|$.}
\label{fig:b2-odd-joining}
\end{figure}

\smallskip
\noindent\textbf{Step 3.  The case that $|P|$ is even.}
Then~\eqref{x1-xp} becomes
$$
x_1<x_2>x_3<x_4>\cdots >x_{|P|-1}<x_{|P|}.
$$ 
There are two alternating paths~\eqref{1-x_1-x_2} and~\eqref{2-x_1-x_2} in $\{x_1,x_2\}\times B_2$, the same as those when $|P|$ is odd.
For $(x_{|P|-1},A)$ and $(x_{|P|},A)$,
there are two alternating paths
\begin{align*}
(x_{|P|-1},\{1,2\})
<&(x_{|P|},\{1,2\})>(x_{|P|-1},\{1\}),\\
(x_{|P|-1},\{2\})
<(x_{|P|},\{2\})>&(x_{|P|},\varnothing)
<(x_{|P|},\{1\})>(x_{|P|-1},\varnothing).
\end{align*}
The four paths visit each vertex $(x_1,A)$ and $(x_{|P|},A)$ 
exactly once.

\smallskip
\noindent\textbf{Step 4. This construction forms an AHC.}
The four paths $I_{\varnothing}, I_{\{1\}},I_{\{2\}}, I_{\{1,2\}}$ visit the vertices 
$(x_2,A),\ldots,(x_{|P|-1},A)$ for each $A\in \{\varnothing, \{1\},\{2\}, \{1,2\}\}$. 
The paths in Step 2 or Step 3 visit the vertices  $(x_1,A)$ and $(x_{|P|},A)$. 
Suppose that $|P|$ is odd.
Then
\begin{align*}
(x_1,\varnothing)&\rightarrow I_{\varnothing} \rightarrow (x_{|P|},\varnothing) \rightarrow \overleftarrow{I_{\{2\}}}
\rightarrow (x_1,\{2\})
<(x_1,\{1,2\})>(x_1,\{1\})\\
&\rightarrow I_{\{1,2\}}\rightarrow (x_{|P|},\{2\})<(x_{|P|},\{1,2\})>(x_{|P|},\{1\})\rightarrow \overleftarrow{I_{\{1\}}}
\rightarrow (x_1,\varnothing)
\end{align*}
forms an AHC in $P\times B_2$.
Similarly, we can obtain an AHC in $P\times B_2$ when $|P|$ is even. 

Suppose that the AHP in $P$ is of the form
\begin{equation}\label{x22-xp}
x_1>x_2<x_3>x_4<\cdots.
\end{equation}
Then the dual $P^*$ admits an AHP of the form~\eqref{x1-xp}.
Thus $P^*\times B_2$ admits an AHC.
By Lemma~\ref{lem:basic-properties} (iii),
\[
(P^\ast\times B_2)^\ast
\cong P\times B_2^\ast
\cong P\times B_2
\]
admits an AHC.
This completes the proof.
\end{proof}

We next consider the case $P\times B_3$.

\begin{lemma}
\label{lem:b3}
Let $P$ be a finite poset with $|P|\ge 2$.
If $P$ admits an AHP, then $P\times B_3$ admits an AHC.
\end{lemma}

\begin{proof}
If $|P|=2$, then $P\cong B_1$ is a two-element chain and $P\times B_3\cong B_4\cong B_2\times B_2$.
Since $B_2$ admits an AHP, $B_4$ admits an AHC by Lemma~\ref{lem:b2}.

Suppose that $|P|\ge3$. 
As in the proof of Lemma~\ref{lem:b2}, we may assume that the AHP in $P$ is of the form
\[
x_1<x_2>x_3<x_4>\cdots.
\]

\smallskip
\noindent\textbf{Step 1. Construct eight disjoint alternating paths.}
For each 
$$A\in\{\varnothing, \{1\},\{2\}, \{3\}, \{1,2\}, \{1,3\},\{2,3\},\{1,2,3\}\},$$ let
\[
I_A=(x_2,A),(x_3,A),\ldots,(x_{|P|-1},A).
\]
These eight paths are disjoint and traverse $\{x_2,\ldots,x_{|P|-1}\}\times B_3$.

\smallskip
\noindent\textbf{Step 2. The case that $|P|$ is odd.}
We construct four joining paths whose endpoints lie in
$\{x_2\}\times B_3$ and whose internal vertices lie in
$\{x_1\}\times B_3$:
\begin{align}
(x_2,\varnothing)>&(x_1,\varnothing)<(x_2,\{1\}),\label{3-x_1-x_2}\\
(x_2,\{3\})>(x_1,\{3\})<&(x_1,\{1,3\})>(x_1,\{1\})<(x_2,\{1,3\}), \label{4-x_1-x_2}\\
(x_2,\{2\})>&(x_1,\{2\})<(x_2,\{2,3\}),\label{5-x_1-x_2}\\
(x_2,\{1,2\})>(x_1,\{1,2\})<&(x_1,\{1,2,3\})>(x_1,\{2,3\})<(x_2,\{1,2,3\}).\label{6-x_1-x_2}
\end{align}
Similarly, we construct four joining paths whose endpoints lie in
$\{x_{|P|-1}\}\times B_3$ and whose internal vertices lie in
$\{x_{|P|}\}\times B_3$:
\begin{align*}
(x_{|P|-1},\varnothing)
>&(x_{|P|},\varnothing)<(x_{|P|-1},\{2\}),\\
(x_{|P|-1},\{1\})
>(x_{|P|},\{1\})<&(x_{|P|},\{1,2\})
 >(x_{|P|},\{2\})<(x_{|P|-1},\{1,2\}),\\
(x_{|P|-1},\{3\})
>&(x_{|P|},\{3\})<(x_{|P|-1},\{2,3\}),\\
(x_{|P|-1},\{1,3\})
>(x_{|P|},\{1,3\})<&(x_{|P|},\{1,2,3\})
 >(x_{|P|},\{2,3\})<(x_{|P|-1},\{1,2,3\}).
\end{align*}
These eight joining paths visit every vertex $(x_1,A)$ and $(x_{|P|},A)$ exactly once as an internal vertex.

\smallskip
\noindent\textbf{Step 3. The case that $|P|$ is even.}
There are four alternating paths~\eqref{3-x_1-x_2},~\eqref{4-x_1-x_2},~\eqref{5-x_1-x_2} and~\eqref{6-x_1-x_2} in $\{x_1,x_2\}\times B_3$, the same as those when $|P|$ is odd.
For $(x_{|P|-1},A)$ and $(x_{|P|},A)$,
there are four alternating paths
\begin{align*}
(x_{|P|-1},\{1,2,3\})
<&(x_{|P|},\{1,2,3\})>(x_{|P|-1},\{1,3\}),\\
(x_{|P|-1},\{2,3\})
<(x_{|P|},\{2,3\})>&(x_{|P|},\{3\})
<(x_{|P|},\{1,3\})>(x_{|P|-1},\{3\}),\\
(x_{|P|-1},\{1,2\})
<&(x_{|P|},\{1,2\})>(x_{|P|-1},\{1\}),\\
(x_{|P|-1},\{2\})
<(x_{|P|},\{2\})>&(x_{|P|},\varnothing)
<(x_{|P|},\{1\})>(x_{|P|-1},\varnothing).
\end{align*}

\smallskip
\noindent\textbf{Step 4. This construction forms an AHC.}
The eight paths and the joining paths visit all vertices of $P\times B_3$.
When $|P|$ is odd, 
\begin{align*}
(x_1,\varnothing)
&\rightarrow I_\varnothing
 \rightarrow(x_{|P|},\varnothing)
 \rightarrow\overleftarrow{I_{\{2\}}}
 \rightarrow(x_1,\{2\})
 \rightarrow I_{\{2,3\}}\\
&\rightarrow(x_{|P|},\{3\})
 \rightarrow\overleftarrow{I_{\{3\}}}
 \rightarrow(x_1,\{3\})
 <(x_1,\{1,3\})>(x_1,\{1\})
 \rightarrow I_{\{1,3\}}\\
&\rightarrow(x_{|P|},\{1,3\})
 <(x_{|P|},\{1,2,3\})>(x_{|P|},\{2,3\})
 \rightarrow\overleftarrow{I_{\{1,2,3\}}}\\
&\rightarrow(x_1,\{2,3\})
 <(x_1,\{1,2,3\})>(x_1,\{1,2\})
 \rightarrow I_{\{1,2\}}\\
&\rightarrow(x_{|P|},\{2\})
 <(x_{|P|},\{1,2\})>(x_{|P|},\{1\})
 \rightarrow\overleftarrow{I_{\{1\}}}
 \rightarrow(x_1,\varnothing)
\end{align*}
forms an AHC.
Similarly, the connections in Step 3 give an AHC
when $|P|$ is even.
\end{proof}




Note that if a finite poset $P$ admits an AHP, then $P\times B_1$ does not necessarily admit an AHC.
We present the following result.

\begin{lemma}
\label{lem:b1}
Let $P$ be a finite poset.
If $P$ admits an AHC, then $P\times B_1$ admits an AHP.
\end{lemma}

\begin{proof}
By Lemma~\ref{lem:basic-properties} (ii), $|P|$ is even.
Suppose that the AHC in $P$ is of the form
\[
x_1<x_2>x_3<\cdots<x_{|P|}>x_1.
\]

Identify $B_1$ with the chain $0<1$.
For each $\varepsilon\in\{0,1\}$, define the alternating path
\[
J_\varepsilon=
(x_2,\varepsilon),(x_3,\varepsilon),\ldots,
(x_{|P|},\varepsilon),(x_1,\varepsilon).
\]
These two paths $J_0$ and $J_1$ are disjoint and traverse $P\times B_1$.
Since $x_1<x_2$ and $0<1$, we have $(x_1,0)<(x_2,1)$.
Thus $J_0\rightarrow J_1$ is the desired AHP.
\end{proof}

Now we are ready to prove our first theorem.
\begin{proof}[Proof of Theorem~\ref{thm:intro-boolean}]
(i). We proceed by induction on $r$.
By Lemmas~\ref{lem:b2} and~\ref{lem:b3},
we have shown that if $P$ admits an AHP, then $P\times B_r$ admits an AHC for $r=2, 3$.
	Let $r\ge4$ and assume that the result holds for $r-2$.
	Then $P\times B_{r-2}$ admits an AHC.
By Lemma~\ref{lem:basic-properties} (i), $P\times B_{r-2}$ admits an AHP.
	Applying Lemma~\ref{lem:b2} to $P\times B_{r-2}$, we obtain an AHC in
	\[
	(P\times B_{r-2})\times B_2
	\cong P\times B_r.
	\]

(ii). 
Suppose that $P$ admits an AHC. 
Then  by Lemma~\ref{lem:basic-properties} (i), $P$ admits an AHP. 
It is trivial that for $r=0$, the product $P\times B_0\cong P$ admits an AHP.
By Lemma~\ref{lem:b1}, the case $r=1$ holds, that is,
$P\times B_1$ admits an AHP.
For $r\ge2$, since $P$ admits an AHP,
by~(i), $P\times B_r$ admits an AHC, and hence admits an AHP. 
\end{proof}

\begin{remark}\label{rem:boolean-bound}
	The restriction $r\ge2$ in Theorem~\ref{thm:intro-boolean}(i) is sharp.
	For example, $P=B_1$ admits an AHP, but $P\times B_1\cong B_2$ does not admit an AHC.
\end{remark}

\begin{corollary}\label{B-AHP}
For every $n\ge 0$, $B_n$ admits an AHP.
\end{corollary}
\begin{proof}
It is trivial that the singleton $B_0$ and the chain $B_1$ admit AHPs.
Clearly, $B_2$ admits an AHP of the form
$\{1\}<\{1,2\}>\varnothing<\{2\}$.
For $n\ge3$, Corollary~\ref{coro:katona} and Lemma~\ref{lem:basic-properties} (i) give an AHP in $B_n$.
\end{proof}

\section{Products with chains}
\label{sec:chains}
In this section, we prove Theorem~\ref{thm:intro-chains}.
Recall that $C_m$ denotes the chain
\[
1<2<\cdots<m+1.
\]
Every chain $C_m$ admits an AHP.
Indeed, the case $m=0$ is immediate.
For $m\ge1$, alternately taking the smallest and largest
remaining elements, starting with the smallest, gives the AHP
\begin{equation}\label{chain-AHP}
1<m+1>2<m>3<m-1>\cdots.
\end{equation}
For odd $m\ge 3$, this also yields an AHC in $C_m$.

We first prove that if $P$ admits an AHP, then $P\times C_3$ admits an AHC.
Then we prove that if $P\times C_m$ admits an AHP (resp. AHC), then so does $P\times C_{m+2}$.
Finally, we prove Theorem~\ref{thm:intro-chains} by induction.

\begin{lemma}
\label{lem:c3}
Let $P$ be a finite poset.
If $P$ admits an AHP, then $P\times C_3$ admits an AHC.
\end{lemma}

\begin{proof}
We may assume that $|P|\ge2$, since $C_3$ admits an AHC.
By Lemma~\ref{lem:b2}, $P\times B_2$ admits an AHC.
Identify the elements $\varnothing,\{1\},\{2\},\{1,2\}$ of $B_2$
with $1,2,3,4$, respectively.
Under this identification, $C_3$ is a refinement of $B_2$.
Since $P\times C_3$ is a refinement of $P\times B_2$, Lemma~\ref{lem:basic-properties}~(iv) implies that $P\times C_3$ admits an AHC.
\end{proof}

\begin{lemma}
\label{lem:chain-extension}
Let $P$ be a finite poset and let $m\ge0$.
If $P\times C_m$ admits an AHP (resp. AHC), then $P\times C_{m+2}$ admits an AHP (resp. AHC).
\end{lemma}

\begin{proof}
We first consider the AHP case. 
Let $\mathcal P$ be an AHP in $P\times C_m$.
If $\mathcal P$ consists of a single vertex $(x,1)$, then
$(x,1)<(x,3)>(x,2)$ is the required AHP.
We now assume that $|\mathcal P|\ge 2$.
First, we give an operation on $\mathcal P$.
\begin{itemize}
    \item For $j\ge 2$, replace every vertex $(x,j)$ in $\mathcal {P}$ by $(x,j+2)$.
    \item For $j=1$, replace $(x,1)$ in the following way:
$$
\begin{cases}
(x,1)<(x,3)>(x,2) \qquad \textrm{when}\quad  \cdots>(x,1)<\cdots \quad  \textrm{in } \, \mathcal P;\\
(x,3)>(x,1)<(x,2) \qquad \textrm{when}\quad \cdots<(x,1)>\cdots \quad \textrm{in } \,\mathcal P.
\end{cases}
$$
\end{itemize}
If $(x,1)$ is an endpoint of $\mathcal P$, use the first
replacement when $(x,1)$ is smaller than its unique neighbor,
and use the second replacement when it is larger than its
unique neighbor.
Applying the operation to every vertex of $\mathcal P$ yields a vertex sequence $\mathcal {Q}$ in $P\times C_{m+2}$.

We now show that $\mathcal Q$ is an AHP in $P\times C_{m+2}$.
Let $(x,i)$ and $(y,j)$ be consecutive in $\mathcal P$.
It suffices to show that this operation preserves the alternating order relation of  $(x,i)$ and $(y,j)$.

\smallskip
\noindent\textbf{Case 1: $i,j\ge 2$.}
We obtain $(x,i+2)$ and $(y,j+2)$ in $\mathcal Q$. 
And $(x,i+2)<(y,j+2)$ if and only if $(x,i)<(y,j)$.
The alternating order relation is preserved.

\smallskip
\noindent\textbf{Case 2: $i=j=1$.}
If $(x,1)<(y,1)$, then $x<y$ in $P$, and the operation gives
\[
(x,1)<(x,3)>(x,2)<(y,3)>(y,1)<(y,2).
\]
Similarly, if $(x,1)>(y,1)$, then $x>y$, and the operation gives
\[
(x,3)>(x,1)<(x,2)>(y,1)<(y,3)>(y,2).
\]
The alternating order relation is preserved.

\smallskip
\noindent\textbf{Case 3: $i=1,j\ge 2$ or $i\ge 2,j=1$.}
Consider $(x,1)$ and $(y,j)$, where $j\ge2$.
Then we have  $(x,1)<(y,j)$, which implies that $x\le y$.
Then 
\[
(x,1)<(x,3)>(x,2)<(y,j+2).
\]
Similarly, consider $(x,j)$ and $(y,1)$, where $j\ge2$.
Then $x\ge y$ and
\[
(x,j+2)>(y,1)<(y,3)>(y,2).
\]
In both cases, the alternating order relation is preserved.

For each $x\in P$, the replacements use each chain coordinate $1,\ldots,m+3$ exactly once.
Thus $\mathcal Q$ visits every vertex of $P\times C_{m+2}$ exactly once and is an AHP.

For the AHC case, we operate on $(x,j)$ in an AHC of $P\times C_m$ in the same way.
Similarly, we can prove that $P\times C_{m+2}$ admits an AHC.
This completes the proof.
\end{proof}

We are ready to prove Theorem~\ref{thm:intro-chains}.

\begin{proof}[Proof of Theorem~\ref{thm:intro-chains}]
By Lemmas~\ref{lem:c3} and~\ref{lem:chain-extension}, if $P$ admits an AHP, then $P\times C_m$ admits an AHC for every odd $m\ge3$.
This proves (ii).

Starting with $P\times C_0\cong P$, repeated applications of Lemma~\ref{lem:chain-extension} give an AHP (resp., AHC) in $P\times C_m$ for every even $m\ge0$ whenever $P$ admits an AHP (resp., AHC).
Together with Lemma~\ref{lem:basic-properties} (i), these observations prove (i) and (iii).

Suppose that $P$ admits an AHC.
By Lemma~\ref{lem:basic-properties} (i), $P$ admits an AHP. 
Applying part (i), we obtain that $P\times C_m$ admits an AHP for every $m\ne 1$.
Since $C_1\cong B_1$, by Lemma~\ref{lem:b1}, $P\times C_1$ admits an AHP.
This proves (iv).





\end{proof}

\begin{remark}
\label{rem:chain-sharpness}
The ranges in (i)--(iii) of Theorem~\ref{thm:intro-chains} are sharp.
For (i), the snake $S_5$ admits an AHP, but $S_5\times C_1$ admits no AHP.
For (ii), taking $P=C_1$ and $m=1$ gives $P\times C_1\cong B_2$, which admits no AHC.
For every even $m\ge0$, take $P=C_2$.
Then $C_2$ admits the AHP $1<3>2$ but $C_2\times C_m$ has $3(m+1)$ elements, an odd number.
By Lemma~\ref{lem:basic-properties} (ii), $C_2\times C_m$ admits no AHC.
Thus the range in part (ii) cannot be enlarged.
For (iii), let $Q$ have elements $a_1,a_2,b_1,b_2$, with the only strict order relations given by $a_i<b_j$ for all $i,j\in\{1,2\}$.
Then $a_1<b_1>a_2<b_2>a_1$ is an AHC in $Q$, but $Q\times C_1$ admits no AHC.
\end{remark}





\section{AHP and AHC in divisor lattices}
\label{sec:products}
In this section, we first prove Theorem~\ref{thm:intro-products}. 
Then we study the weak embeddings of snakes into divisor lattices.

\begin{lemma}
\label{lem:b1-chain}
For every $r\ge1$ and $m\ge2$, $B_r\times C_m$ admits an AHC.
\end{lemma}

\begin{proof}
First consider $r=1$.
Identify $B_1$ with the chain $1<2$.
The product $B_1\times C_2$ admits the AHC
\[
(1,1)<(1,3)>(1,2)<(2,2)>(2,1)<(2,3)>(1,1).
\]
Since $B_1$ admits an AHP, Lemma~\ref{lem:c3} also gives an AHC in $B_1\times C_3$.
Repeated application of Lemma~\ref{lem:chain-extension} gives an AHC in $B_1\times C_m$ for every $m\ge2$.

For $r\ge 2$, since $C_m$ admits an AHP and $|C_m|=m+1\ge 3$, Theorem~\ref{thm:intro-boolean} (i) gives an AHC in $C_m\times B_r$.
\end{proof}

\begin{proof}[Proof of Theorem~\ref{thm:intro-products}]
Write 
\[
D_n\cong B_r\times C_{a_1}\times\cdots\times C_{a_s},\quad \textrm{where}\,\, r,s\ge0\,\,\textrm{and}\,\,  a_i\ge2\,\, \textrm{for all }\, i.
\]

\smallskip
(i). By Corollary~\ref{B-AHP}, $B_r$ admits an AHP.
Since $a_i\ge2$, repeated application of Theorem~\ref{thm:intro-chains} (i) gives an AHP in $D_n$.

\smallskip
(ii).
If $D_n$ admits an AHC, then Lemma~\ref{lem:basic-properties} (ii) implies that $|D_n|$ is even and $|D_n|\ge4$.
Hence $D_n\ncong C_1$.
Since $B_2$ admits no AHC, we also have $D_n\ncong B_2$.

Conversely, suppose that $|D_n|$ is even, $D_n\ncong C_1$, and $D_n\ncong B_2$.
If $r=0$, then
\[
|D_n|=\prod_{i=1}^s(a_i+1).
\]
Since $|D_n|$ is even, some $a_i$ is odd,
and hence $a_i\ge3$.
Write $D_n\cong R\times C_{a_i}$, where $R$ is a divisor lattice.
By (i), $R$ admits an AHP.
Thus Theorem~\ref{thm:intro-chains} (ii) gives an AHC in $D_n$.

Now suppose that $r\ge1$.
If $s=0$, then $D_n\cong B_r$.
The exclusions $D_n\ncong C_1$ and $D_n\ncong B_2$ imply that $r\ge3$.
Hence Corollary~\ref{coro:katona} gives an AHC in $D_n$.

If $s\ge1$, then Lemma~\ref{lem:b1-chain} gives an AHC in $B_r\times C_{a_1}$.
Since $a_i\ge2$ for $2\le i\le s$, repeated application of Theorem~\ref{thm:intro-chains} (iii) gives an AHC in $D_n$.

\end{proof}

Katona~\cite[Section~8]{Katona}
studied the  weak embeddings of snakes into Boolean lattices and 
suggested that $S_k\hookrightarrow B_n$ whenever $k\le2^n$.
We confirm this statement.
More generally,
we derive the following result on the weak embeddings of snakes into divisor lattices.



\begin{corollary}
\label{cor:product-snakes}
Let $D_n$ be the divisor lattice of a positive integer $n$.
For every integer $k\ge1$, the snake $S_k$ weakly embeds into
$D_n$ if and only if $k\le |D_n|$.
\end{corollary}

\begin{proof}
If $S_k\hookrightarrow D_n$, then injectivity gives $k\le |D_n|$.
Conversely, assume $k\le |D_n|$.
By Theorem~\ref{thm:intro-products} (i), $D_n$ admits an AHP.
Since $D_n$ is self-dual, we obtain a weak embedding of $S_k$ into $D_n$.
This completes the proof.
\end{proof}
\begin{remark}
In particular, $S_k\hookrightarrow B_n$ if and only if $k\le2^n$, for all $n\ge0$ and $k\ge1$.
\end{remark}

\section{Applications to snake-free families}
\label{sec:applications}
For a snake $S_k$ defined in~\eqref{def:snake},
a subset of a poset is called \emph{$S_k$-free} if it admits no weak embedding of $S_k$.
Snake-free families are closely related to AHPs.
For a finite poset $P$, define
\[
\La(P,S_k)
=
\max\bigl\{
|\mathcal F|:
\mathcal F\subseteq P
\text{ is }S_k\text{-free}
\bigr\}.
\]
For $P=B_n$, write $\La(n,S_k)=\La(B_n,S_k)$.
Let $w(R)$ denote the maximum size of an antichain in a finite poset $R$.
We obtain the following result.


\begin{proposition}
\label{prop:product-snake-free}
For a divisor lattice $Q$ and a finite poset $R$, let $\mathcal A$ be a nonempty antichain of $R$.
Then the family
\[
\mathcal F=Q\times\mathcal A\subseteq Q\times R
\]
is $S_k$-free if and only if $k>|Q|$.
Moreover, for $k>|Q|$,
\[
\La(Q\times R,S_k)
\ge |Q|\,w(R).
\]
\end{proposition}

\begin{proof}
For each $a\in\mathcal A$, the family $Q\times\{a\}$ is a copy of $Q$.
Since $\mathcal A$ is an antichain, elements from distinct copies are incomparable.

Consecutive elements of $S_k$ are comparable.
Hence, the image of every weak embedding of $S_k$ into $\mathcal F$ lies in a single copy $Q\times\{a\}$.
By Corollary~\ref{cor:product-snakes}, such an embedding
exists if and only if $k\le |Q|$.
Thus $\mathcal F$ is $S_k$-free if and only if $k>|Q|$.

Counting the copies of $Q$ gives $|\mathcal F|=|Q|\,|\mathcal A|$.
Taking $|\mathcal A|=w(R)$ gives the stated lower bound.
\end{proof}

\begin{corollary}
\label{cor:divisor-snake-bound}
Let $n$ be a positive integer, and let $a,b$ be coprime positive integers with $ab\mid n$.
For every integer $k>|D_a|$, we have
\[
\La(D_n,S_k)\ge |D_a|\,w(D_b).
\]
\end{corollary}

\begin{proof}
Let $a,b$ be positive integers satisfying
$ab\mid n$, $\gcd(a,b)=1$, and $|D_a|<k$.
By coprimality, we have $D_a\times D_b\cong D_{ab}$.
Moreover, $D_{ab}$ is an induced subposet of $D_n$.

Let $\mathcal A$ be a maximum antichain of $D_b$.
By Proposition~\ref{prop:product-snake-free},
$D_a\times\mathcal A$ is $S_k$-free and has
$|D_a|\,w(D_b)$ elements.
This gives an $S_k$-free subset of $D_n$ with $|D_a|w(D_b)$ elements.
\end{proof}

For Boolean lattices, this recovers the construction of Katona~\cite[Proposition~8.2]{Katona}.
Fix $0\le\ell\le n$ and an $\ell$-element set $T\subseteq[n]$.
Taking $Q=B_\ell$, $R=B_{n-\ell}$, and $\mathcal A$ to be a middle level of $R$ gives
\[
\mathcal F_{n,\ell}
=
\left\{
A\cup C:
A\subseteq T,\quad
C\subseteq[n]\setminus T,\quad
|C|=\left\lfloor\frac{n-\ell}{2}\right\rfloor
\right\}.
\]
By Proposition~\ref{prop:product-snake-free}, $\mathcal F_{n,\ell}$ is $S_k$-free exactly when $k>2^\ell$, which confirms the applicability condition discussed by Katona~\cite[Section~8]{Katona}.
Its size is
\[
|\mathcal F_{n,\ell}|
=
2^\ell
\binom{n-\ell}{\lfloor(n-\ell)/2\rfloor}.
\]
For $k\ge2$, choosing
\[
\ell=\min\{n,\lfloor\log_2(k-1)\rfloor\}
\]
therefore gives
\[
\La(n,S_k)
\ge
2^\ell
\binom{n-\ell}{\lfloor(n-\ell)/2\rfloor}.
\]

\end{document}